\documentclass[11pt]{amsart}
\usepackage{graphicx}
\usepackage{latexsym}
\usepackage{amsfonts,amsmath,amssymb}
\newtheorem{theorem}{Theorem}[section]

\newtheorem{lemma}[theorem]{Lemma}

\def\ex{\textup{E\/}}

\def\eps{\varepsilon}
\def\la{\lambda}
\def\a{\alpha}
\def\be{\beta}

\def\ga{\gamma}
\def\part{\partial}

\usepackage{xcolor}

\newcommand{\beq}{\begin{equation}}
\newcommand{\eeq}{\end{equation}}

\theoremstyle{remark}

\numberwithin{equation}{section}
\date{\today}
\begin{document}

\title[Stable matchings]{On random stable matchings:  cyclic matchings with strict preferences and two-side matchings with partially ordered preferences.}

\author{Boris Pittel}
\address{Department of Mathematics, The Ohio State University, Columbus, Ohio 43210, USA}
\email{bgp@math.ohio-state.edu}



\begin{abstract} Consider a cyclically ordered collection of $r$ equi-numerous agent sets with strict
preferences of every agent over the agents from the next agent set. A weakly stable cyclic matching is a partition of the set of agents into disjoint union of $r$-long
cycles, one agent from each set per cycle, such that there are no destabilizing $r$-long cycles, i.e. cycles in which every
agent strictly prefers its successor to its successor in the matching. Assuming that the preferences are uniformly random and 
independent, we show that the expected number of stable matchings grows with $n$ (cardinality of each agent set) as $(n\log n)^{r-1}$. 
We also consider a bipartite stable matching
problem where preference list of each agent 
 forms a partially ordered set. Each partial order
is an intersection of several, $k_i$ for side $i$, independent, uniformly random, strict orders. 
For $k_1+k_2>2$, the expected number of stable matchings is analyzed for three, progressively stronger, notions of stability.
The expected number of weakly stable matchings is shown to grow super-exponentially fast. In contrast, for $\min(k_1,k_2)>1$, the fraction of instances 
with at least one strongly stable (super-stable) matching is super-exponentially small.
\end{abstract}
\maketitle

\section{Introduction and main results} Consider the set of $n$ men and $n$ women facing a problem of selecting a marriage partner. A  marriage $M$ is a matching (bijection) between the two sets.  It is assumed that each man and each woman ranks all the members as a potential marriage partner, with no ties allowed.  A marriage is called stable if there is
no unmarried pair--a man and a woman--who prefer each other to their respective partners in the marriage. A classic theorem, due to Gale and Shapley \cite{GalSha}, asserts that, given any system of preferences, there exists at least one stable marriage $M$. 

The proof of this theorem is algorithmic.
A bijection is constructed in steps such that at each step every man not currently on hold makes
a proposal to his best choice among women who haven't rejected him before, and the chosen woman either provisionally puts the man on hold or rejects him, based on comparison of him to her  current suitor if she has one already. Since a woman who once gets proposed to always has a man on hold
afterwards, after finally many steps every woman has a suitor, and the resulting bijection turns out to be stable. Of course the roles can be reversed, with women proposing and each man selecting  between the current proponent and a woman whose proposal he currently holds, if there is such a woman.
In general, the two resulting matchings, $M_1$ and $M_2$ are different, one man-optimal, another woman-optimal. 

The interested reader is encouraged to consult Gusfield and Irving \cite{GusIrv} 
for
a rich, detailed analysis of the algebraic (lattice) structure of stable matchings set, and Manlove \cite{Man} for encyclopedic presentation of
a growing body of a contemporary research on a diverse  variety of matching problems. 

A decade after the Gale-Shapley paper, McVitie and Wilson \cite{McVWil} developed an alternative, sequential, algorithm in which proposals by one side to another are made one at a time. This procedure delivers the same
matching as the Gale-Shapley algorithm. 
This purely combinatorial, numbers-free, description  
begs for a probabilistic analysis of the problem chosen uniformly at random among all the instances, whose total number is $(n!)^{2n}$. In a pioneering paper \cite{Wil} Wilson reduced
the work of the sequential algorithm to a classic urn scheme (coupon-collector problem) and proved that the expected running time, whence the expected total rank of wives in the man-optimal
matching, is at most $nH_n\sim n\log n$, $H_n=\sum_{j=1}^n 1/j$.

A few years later, Knuth \cite{Knu}, among other results,  found a better upper bound $(n-1)H_n+1$, and established a matching lower bound $nH_n-O(\log^4n)$. He also posed a series of open problems, one of them
on the {\it expected\/} number of the stable matchings. Knuth pointed out that an answer might be found via his formula for the probability $P(n)$ that a generic matching $M$ is stable:
\begin{equation}\label{Pn=}
P(n)=\overbrace {\idotsint}^{2n}_{\bold x,\,\bold y\in [0,1]^n}\,\prod_{1\le i\neq j\le n}
(1-x_iy_j)\, d\bold x d\bold y.
\end{equation}
(His proof relied on an inclusion-exclusion formula, and an interpretation of each summand as the value of a $2n$-dimensional integral, with the integrand equal to the corresponding summand in the expansion of the integrand in \eqref{Pn=}.) The expected value of $S(n)$, the total number of stable matchings, would then be determined from $\ex[S(n)]=n! P(n)$. 

Following Don Knuth's suggestion, in \cite{Pit1} we used the equation \eqref{Pn=} to obtain an asymptotic formula 
\begin{equation}\label{Pnsim}
P(n)=(1+o(1))\frac{e^{-1}n\log n}{n!}\Longrightarrow \ex[S(n)]\sim e^{-1}n\log n.
\end{equation}
More generally, in \cite{Pit2} we derived a formula for $P_{k,\ell}(n)$, the probability that the generic matching $M$ is stable and $Q_M=k$, $R_M=\ell$; here $Q_M$ ($R_M$,  resp.) is the total
rank of wives as ranked by their husbands (the total rank of husbands as ranked by their wives, resp.) in $M$.

The key element of the proofs of the integral representations for these probabilities, which also imply
the Knuth formula \eqref{Pn=}, was a refined probability space. Its sample point is a pair of two
$n\times n$ matrices $\bold X=\{X_{i,j}\}$, $\bold Y=\{Y_{i,j}\}$ with all $2n^2$  entries being independent, $[0,1]$-uniform random variables. Reading each row of $\bold X$ and each column of $\bold Y$ in increasing order we recover the independent, uniform preferences of each of $n$
men and of each of $n$ women respectively. And, for instance, the integrand in \eqref{Pn=} turns out equal to the
probability that a generic matching $M$ is stable, {\it conditioned\/} on the values $x_i=X_{i,M(i)}$,
$y_j=Y_{M^{-1}(j),j}$

Using the formula for $P_{k,\ell}(n)$,
we proved a {\it law of hyperbola\/}: asymptotically almost surely (a.a.s.)
 $\max_{M}|n^{-3}Q_M R_M-1|\le n^{-1/4+o(1)}$. 
 It followed, in particular, that a.a.s.  $S(n)\ge n^{1/2-o(1)}$, a significant  improvement of  the logarithmic bound in Knuth, Motwani and Pittel \cite{KnuMotPit}. Thus, for a large number of participants,  a typical instance of the preferences has multiple stable
matchings, very nearly obeying the preservation law for the product of the total man-rank and
the total woman-rank. In a way this law is not unlike thermodynamic laws in physics of gases. However those laws are usually of phenomenological nature, while the product law is a rigorous
corollary of the {\it local\/} stability conditions for the random instance of the preferences.

Later Lennon and Pittel \cite{LenPit} extended the techniques in \cite{Pit1}, \cite{Pit2} to
show that $\ex[S(n)^2]\sim (e^{-2}+0.5e^{-3})(n\log n)^2$. Combined with \eqref{Pnsim}, this
result implied that $S(n)$ is of order $n\log n$ with probability $0.84$, at least. Jointly with Shepp
and Veklerov \cite{Pit3} we proved that, for a fixed $k$, the expected number of women with $k$ stable husbands is asymptotic to $(\log n)^{k+1}/$ $(k-1)!$. In a recent ground-breaking paper Ashlagi, Kanoria and Leshno \cite{Ash} analyzed a random stable matching problem with unequal
sizes $n_1$ and $n_2$ of the two sides. They discovered that  the set  of properly scaled total ranks $\{(Q_M,R_M)\}_{M}$ a.a.s. converges to a single point even if $|n_2-n_1|=1$. We showed \cite{Pit4} that  if $n_2\gg n_2-n_1>0$ then  the expected number of stable matchings is asymptotic to 
$e^{-1}n_1/[(n_2-n_1)\log n_1]$, compared to $e^{-1}n_1\log n_1$ for $n_2=n_1$, \cite{Pit1}.


Among other avenues of research, Knuth \cite{Knu} was interested in  whether the stable matching problem can be fruitfully generalized to, say, three 
equinumerous sets of agents, referred to as men, women and dogs. The goal is to partition the sets of agents into triples,
(a man, a woman, a dog), such that [given the agents preferences] the set of triples is stable in some sense. 

As reported in Ng and Hirschberg \cite{NgH}, Knuth was particularly interested in the case of cyclic preferences over individual agents: each of $n$ men ranks $n$ women, and women only; each
of $n$ women ranks $n$ dogs, and dogs only; each of $n$ dogs ranks $n$ men, and men only. We will use the notation ``$w_1\overset m\prec w_2$'' to indicate that a man $m$ prefers woman $w_1$ to $w_2$, and will indicate
similarly the preferences of men and dogs.

Denoting the three sets $\mathcal M$, $\mathcal W$ and $\mathcal D$, a  matching $M$ is a partition of $\mathcal M\cup\mathcal W\cup\mathcal D$ into $n$ ordered triples $(m, w, d)$,
each triple being viewed as a directed $3$-cycle $m\to w\to d\to m$.  So $M$ is a permutation of $\mathcal M\cup\mathcal W\cup\mathcal D$
with cycles of length $3$, each cycle of the same type $m\to w\to d\to m$. For each agent $a$, let $M(a)$ denote the successor of $a$ in the cycle of $M$ that
contains $a$. A directed $t=(m\to w\to d\to m)$ strongly blocks $M$ if  $w\overset m\prec M(m),\, d\overset w\prec M(w),\, m\overset d\prec M(d)$.
A matching $M$ is
called weakly stable if no cycle $t=(m\to w\to d\to m)$ strongly blocks $M$. Eriksson et al. \cite{Eri} and Escamocher and O'Sullivan \cite{Esc} conjectured that a weakly stable matching exists for every instance of cyclic preferences. Recently Lam and Paxton \cite{Lam} found an instance of cyclic preferences that has no weakly stable matching.

Similarly, assuming strict preferences, $t=(m\to w\to d\to m)$ weakly blocks $M$ if at least two agents $a_1$ and $a_2$ from $t$ prefer their successors in $t$ to their successors in $M$, and the successor of $a_3$ in $t$ is $M(a_3)$. $M$ is called strongly stable if no $t$ weakly blocks $M$. 

In this paper we consider the stable matchings under cyclic preferences with $r\ge 3$ sides. Here we have an ordered sequence of $r$ equinumerous agent sets $\mathcal A_1,\dots, \mathcal A_r$;  each $a\in \mathcal A_s$ ranks all agents $a'\in \mathcal A_{s+1}$, $(\mathcal A_{r+1}:=\mathcal A_1)$. A matching $M$ is a partition of $\mathcal A_1\cup\cdots\cup \mathcal A_r$ into disjoint directed cycles $a_1\to\cdots\to a_r\to a_1$, $(a_s \in \mathcal A_s)$. 

$M$ is called weakly stable (w-stable) if there is no tuple $a_1,\dots, a_r$ of agents from the
sets $\mathcal A_1,\dots,A_r$ such that 
\[
a_2\overset {a_1}\prec M(a_1),\, a_3\overset {a_2}\prec M(a_2),\cdots, a_1\overset {a_r}\prec M(a_r),
\]
where $M(a)$ is the successor of $a$ in the cycle from the matching $M$ that contains $a$.

Given $m\ge \lceil r/2\rceil$ (the least integer strictly exceeding $r/2$), we say that a cyclic tuple of agents $a_1\to\cdots\to a_r\to a_1$ weakly blocks
$M$ if at least $m$ agents $a_i$ (strictly) prefer their successors in the tuple to their successors in $M$. We call $M$ strongly stable if no cyclic tuple
weakly blocks $M$.

We will prove
\begin{theorem}\label{newE(Sn)>} Let $S_{n,r}$ denote the total number of weakly stable matchings. For $r\ge 3$, we have
\[
\Bbb E [S_{n,r}]\gtrsim (c_r+o(1))\,\left(\frac{n\log n}{2}\right)^{r-1},\quad c_r:= \Bbb P(T_{r-1}\in [1,2]),
\]
where $T_{r-1}=\sum_{j\in [r-1]}Y_j$, and $Y_j$ are independent $[0,1]$-uniform random variables.
\end{theorem}
\noindent {\bf Note.\/} {\bf (1)\/}  Notice that $c_r >0$ for $r\ge 3$, but, formally, $c_2=0$. As we mentioned, we had proved that $\ex[S_{n,2}]\sim e^{-1}n\log n$, \cite{Pit1}, and also that w.h.p. $S_{n,2}\ge n^{1/2-o(1)}$. Theorem  \ref{newE(Sn)>} emboldens us to conjecture that, for $r>2$, w.h.p. $S_{n,r}\ge
n^{\ga_r}$, where $\ga_r\to\infty$ as $r$ grows. {\bf (2)\/} We worked out a lower bound for the expected number of
strongly stable matchings as well:  the bound fast goes to zero as $n\to\infty$. We conjecture that the fraction of instances with at least one strongly stable matching is vanishing as $n\to\infty$.\\

Gusfield and Irving \cite{GusIrv} and Irving \cite{Irv} introduced and studied a more general stable marriage problem, with two sides of size $n$ each, when the preference
lists are partially ordered, i.e. tied entries in the lists are allowed. They defined three, progressively weaker, notions of a stable
matching. {\bf (1)\/} A matching $M$ is super-stable if there is no unmatched (man, woman) pair such that the man and the woman like each other at least as
much as his/her partner in $M$. {\bf (2)\/} 
 $M$ is strongly stable if there is no unmatched (man, woman) pair such that (a) either the man strictly prefers the woman to his partner in $M$ or the woman and his partner are tied in his list and (b) either the woman strictly prefers the man to her partner in $M$
or the man and her partner in $M$ are tied in her list.  {\bf (3)\/} $M$ is weakly stable if there is no unmatched (man, woman) pair such that they strictly prefer each other to their partners under $M$. It was demonstrated in \cite{GusIrv} that a properly extended Gale-Shapley algorithm allows to determine the existence of a super-stable or strongly stable matching for any given instance of partially ordered preferences. As for weak stability, a stable matching can be found by, first, extending each of the partial orders to a linear order, which can always be done, and in multiple ways, and, second, applying the Gale-Shapley algorithm.

The ground-breaking work in \cite{GusIrv} and \cite{Irv} stimulated an impressive research on stable matchings with partial information
about the preference lists, see for instance  Rastagari et al. \cite{Ras1}, Rastagari et al. \cite{Ras2}, and Gelain et al. \cite{Gel}. What the author learned 
about the contemporary state of research in stable matchings under preferences came from reading the book by Manlove \cite{Man}, a remarkably systematic, thought-provoking, expert survey of highly diverse models and algorithms of stable matchings.
  
It occurred to us that as a natural,
more
general, version of the sequence of $n$ independent $[0,1]$-Uniforms, inducing the uniform linearly ordered preference list, one can consider 
the sequence of $n$ independent $[0,1]^k$-Uniforms, i.e. a sequence of $n$ points $\{\bold Z^{(j)}\}_{j\in [n]}$ chosen uniformly, and independently,
from the $k$-cube $[0,1]^k$. Intersecting the $k$ coordinate linear orders on  $\{\bold Z^{(j)}\}_{j\in [n]}$ we obtain $\mathcal P_k(n)$, a partial order on those $n$ points, with order dimension $k$ at most. 
As Brightwell indicated in his authoritative survey \cite{Bri},  the random $k$-dimensional partial orders have been a subject of systematic studies since a $1985$ paper of Winkler \cite{Win1}, see also Winkler \cite{Win2}, and ``had in fact occurred in various different guises earlier''. 

So in this paper we consider the case when the partially ordered preferences of $n$ men (women resp.) over women (men resp.) are $n$ independent copies of the random partial order $\mathcal P_{k_1}(n)$ ($\mathcal P_{k_2}(n)$ resp.). Equivalently, we have the $n$ sequences $\{\bold X_i^{(j)}\}_{j\in [n]}$, $(i\in [n])$, and the $n$ sequences $\{\bold Y_j^{(i)}\}_{i\in [n]}$, $(j\in [n])$, with all $\bold X_i^{(j)}$ and $\bold Y_j^{(i)}$ being independent,
$[0,1]^{k_1}$-Uniforms and $[0,1]^{k_2}$-Uniforms respectively, so that, for instance, the partial order induced by $\{\bold X_i^{(j)}\}_{j\in [n]}$ is the partially ordered
preference list of man $i$ over the set of women $j\in [n]$.

Let $S_{n,w}$, $S_{n, s}$ and $S_{n,\text{sup}}$ stand for the random number of weakly,  strongly and super-stable matchings respectively. Obviously
$S_{n,w} \ge S_{n, s}\ge S_{n,\text{sup}}$. Also $S_{n,w}$ ($S_{n,\text{sup}}$ resp.) is stochastically monotone increasing (decreasing resp.)
with respect to $(k_1,k_2)$. We will prove two claims.
\begin{theorem}\label{weak} Let $k:=\max(k_1,k_2)>1$. Then
\[
(\log n)^{(k-1-o(1))n} \le \Bbb E[S_{n,w}]\le (\log n)^{(k_1+k_2-1+o(1))n}.
\]
\end{theorem}

\begin{theorem}\label{strong/super} {\bf (i)\/} If $k:=\min(k_1,k_2)>2$, then 
\[
\Bbb E[S_{n,s}] \le n^{-n\bigl(\frac{k-1}{k+1}-o(1)\bigr)},\quad \Bbb E[S_{n,\text{sup}}] \le n^{-n\bigl(k-1-o(1)\bigr)}.
\]
{\bf (ii)\/} If $\min(k_1,k_2)=1$ and  $k:=\max(k_1,k_2)>1$, then
\begin{align*}
&\,[e^2(k-1)!]^{-n} \le \Bbb E[S_{n,s}] = \Bbb E[S_{n,\text{sup}}] \le (\rho_k+o(1))^n,\\
&\rho_k:=e^{-1}\int_0^1e^{z/2} z^{-1}\frac{\log^{k-1}(1-z)^{-1}}{(k-1)!}\,dz\le \rho_2<0.83.
\end{align*}
\end{theorem}
Thus, for $\max(k_1,k_2)>1$, the weakly stable matchings are, on average, super-exponentially
numerous. 
In stark contrast, when $\min(k_1,k_2)>1$ as well, the fractions of problem instances with at least one
strongly stable or super-stable matchings vanish at a super-exponential rate, particularly fast in the case of super-stability.
Even in the most favorable case, when  $\min(k_1,k_2)=1$, the fraction of solvable instances, for either strongly stable or super-stable matchings, is exponentially small.


\section{Cyclic stable matchings.} Suppose  the $n$ {\it complete\/} preference lists for agents in $\mathcal A_{s+1}$ by agents in $\mathcal A_s$ ($s\in [r]$,\,$\mathcal A_{r+1}:=\mathcal A_r$), $r n$ lists in total, are chosen uniformly at random and independently of each other. For simplicity
of notations, the sets $\mathcal A_s$ are copies of the set $[n]$.
\begin{lemma}\label{P(M)} Let $M$ be an arbitrary matching on $\cup_{s\in [r]} \mathcal A_s$, and let $\Bbb P(M)$ be the probability that $M$ is weakly stable. Introducing $\bold x^{(s)}=\{x^{(s)}_1,\dots, x^{(s)}_n\}\in [0,1]^n$,  $(s\in [r])$, we have
\begin{equation}\label{P(M)>}
\begin{aligned}
&\Bbb P(M)\ge \idotsint\limits_{\bold x^{(1)},\dots, \bold x^{(r)}\in [0,1]^n}\!\!\! F(\bold x^{(1)},\dots, \bold x^{(r)})\,\,d \bold x^{(1)}\,\cdots,\,d \bold x^{(r)},\\
&\qquad F(\bold x^{(1)},\dots, \bold x^{(r)})=\prod_{i_1,\dots, i_r\in [n]\atop i_1\neq\cdots\neq i_r\neq i_1}\Bigl(1-\prod_{s\in [r]}x_{i_s}^{(s)}\Bigr).
\end{aligned}
\end{equation}
\end{lemma}
\noindent {\bf Note.\/}  This Lemma extends an {\it identity\/} 
\begin{equation}\label{Knu}
\Bbb P(M)= \idotsint\limits_{\bold x, \bold y\in [0,1]^{n}}\,\prod_{i\neq j}\,\,\bigl(1-x_i y_j\bigr)\,d\bold x\,d\bold y,
\end{equation}
found by Knuth \cite{Knu} for the case $r=2$, i.e. the bipartite matchings. An alternative derivation of \eqref{Knu} was given later in \cite{Pit1}. The proof of \eqref{P(M)>} is patterned after that in \cite{Pit1}.

\begin{proof}   By the definition of weak stability, we have
\[
\{M\text{ w-stable}\}=\!\!\bigcap_{a_s\in \mathcal A_s,\,s\in [r]\atop M(a_s)\neq a_{s+1}}\Bigl\{a_2\overset {a_1}\prec M(a_1),\, a_3\overset {a_2}\prec M(a_2),\cdots, a_1\overset {a_r}\prec M(a_r)\Bigr\}^c.
\]
In words, there is no cyclic tuple $a_1\to a_2\to\cdots\to a_r\to a_1$, such that each $a_s$ strictly prefers $a_{s+1}$ to his successor $M(a_s)$ in $M$.
By symmetry,
$\Bbb P(M)$ does not depend on the choice of $M$. So we consider $M$ consisting of $r$-tuples $(i,\dots,i)$, $i\in [n]$.

To lower-bound $\Bbb P(M)$, we refine the probability space of the $r n$ independent uniform preferences. Let $X^{(s)}=\{X_{i,j}^{(s)}\}$, ($s\in [r]$),
 be the $n\times n$ matrices whose $r n^2$ entries are independent $[0,1]$-Uniforms. Reading the entries, from row $1$ to row $n$, in every one of $rn$ rows in the increasing order, starting with $X^{(1)}$ and ending with $X^{(r)}$,  we obtain the preference lists of agents from $\mathcal A_s$ for agents
in $A_{s+1}$, $s\in [r]$. These lists are clearly independent and uniform. Therefore
\begin{align*}
\{M\text{ w-stable}\}=\bigcap_{i_1,\dots, i_r\in [n]\atop i_1\neq i_2\neq\cdots\neq i_r\neq i_1}\!\!\!\Bigl\{X^{(1)}_{i_1,i_2}<X^{(1)}_{i_1,i_1},&\,X^{(2)}_{i_2,i_3}<X^{(2)}_{i_2,i_2},\\
&\,\dots,\, X^{(r)}_{i_r,i_1}<X^{(r)}_{i_r,i_r}\Bigr\}^c.
\end{align*}
Given $\bold x^{(1)},\dots,\bold x^{(r)}\in [0,1]^n$, we will use  ``$|\circ$'' to indicate conditioning on the event $\bigcap_{s\in [r], i\in [n]}\{X^{(s)}_{i,i}=x_i^{(s)}\}$.
Since $X_{i,j}^{(s)}$, $(i\neq j)$, remain independent, $[0,1]$-Uniforms upon conditioning ``$|\circ$'', we have then 
\begin{align*}
&\Bbb P\left(\Bigl\{X^{(1)}_{i_1,i_2}<X^{(1)}_{i_1,i_1},\,
X^{(2)}_{i_2,i_3}<X^{(2)}_{i_2,i_2},\,\dots,\, X^{(r)}_{i_r,i_1}<X^{(r)}_{i_r,i_r}\Bigr\}^c\Big|\circ\right)\\
&=\Bbb P\left(\Bigl\{X^{(1)}_{i_1,i_2}< x_{i_1}^{(1)},\,X^{(2)}_{i_2,i_3}< x_{i_2}^{(2)},\,\dots,\, X^{(r)}_{i_r,i_1}< x_{i_r}^{(r)}\Bigr\}^c\right)\\\
&\qquad\qquad\qquad\quad = 1 - \prod_{s\in [r]}x_{i_s} ^{(s)}.
\end{align*}
As for the events $\mathcal E(\bold i):=\Bigl\{X^{(1)}_{i_1,i_2}< x_{i_1}^{(1)},\,X^{(2)}_{i_2,i_3}< x_{i_2}^{(2)},\,\dots,\, X^{(r)}_{i_r,i_1}< x_{i_r}^{(r)}\Bigr\}^c$,
$(i_1\neq i_2\neq\cdots\neq i_r\neq i_1)$, they are interdependent: $\mathcal E(\bold i)$ and  $\mathcal E(\bold i')$ are independent only if the cyclic tuples $\bold i$ and $\bold i'$ do not share a common edge. Fortunately, each of the events $\mathcal E(\bold i)$ is monotone increasing with
respect to $X_{i_1,i_2}^{(i_1)},\dots, X_{i_r,i_1}^{(r)}$. Since all $X_{i,j}^{(s)}$, $(\i\neq j)$, remain independent upon conditioning $|\circ$,
 the events $\mathcal E(\bold i)$ are {\it positively associated\/} (see Grimmett and Stirzaker \cite{Gri}), yielding
\[
\Bbb P(M\text{ is stable}|\circ)\ge \prod_{i_1,\dots, i_r\in [n]\atop i_1\neq\cdots\neq i_r\neq i_1}\Bigl(1-\prod_{s\in [r]}x_{i_s}^{(s)}\Bigr).
\]
Taking expectations of both sides of this inequality we come to \eqref{P(M)>}.
\end{proof}
We use Lemma \ref{P(M)} to prove
\begin{theorem}\label{newE(Sn)>} Let $S_n$ denote the total number of weakly stable matchings. Then
\[
\Bbb E [S_n]\gtrsim (c_r+o(1))\,\left(\frac{n\log n}{2}\right)^{r-1},\quad c_r:= \Bbb P(T_{r-1}\in [1,2]),
\]
where $T_{r-1}=\sum_{j\in [r-1]}Y_j$, and $Y_j$ are independent $[0,1]$-uniform random variables.
\end{theorem}
\begin{proof}   
To lower-bound the integral in \eqref{P(M)>}, we switch to the $rn$ variables $\xi^{(s)}, \{u_i^{(s)}\}_{i\in [n-1]},\,(s\in [r])$: 
\begin{equation}\label{transnew}
\xi^{(s)} =\sum_{i\in [n]} x_i^{(s)},\quad u_i^{(s)}= \frac{x_i^{(s)}}{\xi^{(s)}},\,\,\,( i\in [n-1]). 
\end{equation}
We introduce the missing $u_n^{(s)}=\frac{x_n^{(s)}}{\xi^{(s)}}$ and get
\begin{equation}\label{clearnew}
\sum_{i\in [n]}u_i^{(s)}=1, \quad \xi^{(s)}\le n,\quad u_i^{(s)}\xi^{(s)}\le 1. 
\end{equation}
The Jacobian of the inverse transformation is $\Bigl(\prod_{s\in [r]}\xi^{(s)}\Bigr)^{n-1}$.  We consider
the subset $\Omega(n)$ of the region in \eqref{clearnew} where
\begin{align}
& n^{\a}\le \xi^{(s)}\le n^{\be},\quad (0<\a<\be<1/2),\label{1Omnew}\\
&\quad\, u_i^{(s)}\le \frac{2\log n}{n},\quad i\in [n],\,s\in [r], \label{2Omnew}\\
&\qquad\qquad\sum_{i\in [n]} (u_i^{(s)})^2\le \frac{3}{n}.\label{3Omnew}
\end{align}
The bounds \eqref{1Omnew}, \eqref{2Omnew} imply the inequalities in \eqref{clearnew}, with plenty of room to spare. For $(\xi^{(s)},\bold u^{(s)})_{s\in [r]}$ meeting
\eqref{1Omnew}, \eqref{2Omnew}, we have
\[
\prod_{s\in [r]} x_{i_s}^{(s)}=\prod_{s\in [r]} \xi^{(s)}u_i^{(s)}\le \left(n^{\be}\cdot \frac{2\log n}{n}\right)^r =O\bigl(n^{-r(1-\be)}\log^r n\bigr)=o(1).
\]
So, using also $\log (1-z)=-z+O(z^2)$, $(z\to 0)$, and \eqref{3Omnew}, we obtain
\begin{equation*}
\begin{aligned}
\sum_{i_1,\dots, i_r\in [n]\atop i_1\neq\cdots\neq i_r\neq i_1}\log\Bigl(1-&\prod_{s\in [r]}x_{i_s}^{(s)}\Bigr)\ge
\sum_{i_1,\dots, i_r\in [n]}\log\Bigl(1-\prod_{s\in [r]}x_{i_s}^{(s)}\Bigr)\\
&=-\sum_{i_1,\dots, i_r\in [n]}\prod_{s\in [r]}x_{i_s}^{(s)}+\sum_{i_1,\dots, i_r\in [n]} O\left(\prod_{s\in [r]}\bigl(x_{i_s}^{(s)}\bigr)^2\right)\\
&=-\prod_{s\in [r]}\xi^{(s)}+O\bigl(n^{-r(1-2\be)}\bigr)=-\prod_{s\in [r]}\xi^{(s)}+o(1),
\end{aligned}
\end{equation*}
since $\be<1/2$. The top inequality comes from dropping the constraint ''$i_1\neq\cdots\neq i_r\neq i_1$''. Our focus will be on $\prod_s \xi^{(s)}=\Theta(n)$,
in which case the conditions $n^{\a}\le \xi^{(s)}$ and \eqref{2Omnew} imply that the resulting additive difference is of order $n^{-\a}\log n=o(1)$.

Therefore, by \ref{P(M)>}, we have: 
\begin{equation}
\begin{aligned}\label{newP(M)>prod}
\Bbb P(M)&\ge (1+o(1))\!\!\!\!\!\!\!\idotsint\limits_{\{\xi^{(s)}\}_{s\in [r]}\text{ meet }\eqref{1Omnew}}\!\!\!\!\!\!\!\!\exp\Biggl(\!\!-\prod_{s\in [r]}\xi^{(s)}\!\!\Biggr)\!\Biggl(\prod_{s\in [r]}\xi^{(s)}\!\!\Biggr)^{n-1}\!\!\!\!\prod_{s\in [r]} d\xi^{(s)}\\
&\qquad\times \idotsint\limits_{\{\bold u^{(s)}\}_{s\in [r]}\text{ meet }\eqref{2Omnew}, \eqref{3Omnew}}\!\!\! 1\,\cdot\,\prod_{s\in [r]}d\bold u^{(s)}.
\end{aligned}
\end{equation}
{\bf (1)\/} Consider the first integral. Keeping the variables $\boldsymbol{\xi}:=\{\xi^{(s)}\}_{s\in [r-1]}$, we
introduce  $\sigma:=\left(\prod_{s\in [r]}\xi^{(s)}\right)^{1/r}$ instead of $\xi^{(r)}$.  The corresponding Jacobian is $r\sigma^{r-1}/\prod_{s\in [r-1]}\xi^{(s)}$. The integral becomes
\begin{equation}\label{Int1}
\idotsint\limits_{\boldsymbol\xi\in [n^{\a}, n^{\be}]^{r-1}}\prod_{s\in [r-1]}\frac{d \xi^{(s)}}{\xi^{(s)}}\int\limits_{\left(n^{\a}\prod_{s\in [r-1]}\xi^{(s)}\right)^{1/r}}
^{\left(n^{\be}\prod_{s\in [r-1]}\xi^{(s)}\right)^{1/r}}\!\!\!\!\!\!\!\!\!\!\!\!\!\! r\,\sigma^{r-1}\,e^{-\sigma^r+r(n-1)\log\sigma}\,d\sigma.
\end{equation}
$\phi(\sigma):=-\sigma^r+r(n-1)\log\sigma$ is concave, and attains its absolute maximum at $\sigma^*:=(n-1)^{1/r}$, so that $\phi(\sigma^*)=-(n-1)+(n-1)\log(n-1)$. Now
\[
\phi''(\sigma)= -r(r-1)\sigma^{r-2} -\frac{r(n-1)}{\sigma^2} \Longrightarrow    \phi''(\sigma^*) =-r^2 (n-1)^{\frac{r-2}{r}}.
\]
Since $r\ge 3$, we have $\phi^{''}(\sigma)=-\Theta(n^{\frac{r-2}{r}})\to -\infty$, for $\sigma\in [0.5 \sigma^*, 2\sigma^*]$. It follows, by the standard Gaussian approximation,  that if
\begin{equation}\label{sigmas}
(\sigma^* -\sigma_1)n^{\frac{r-2}{2r}},\quad (\sigma_2-\sigma^*)n^{\frac{r-2}{2r}}\to\infty,
\end{equation}
then
\begin{multline*}
 \int_{\sigma_1}^{\sigma_2} r\,\sigma^{r-1}\exp\bigl(-\sigma^r+r(n-1)\log\sigma\bigr)\,d\sigma\\
 =(1+o(1))r (\sigma^*)^{r-1} e^{\phi(\sigma^*)}\sqrt{\frac{2\pi}{-\phi^{''}(\sigma^*)}}\\
 =(1+o(1))\sqrt{2\pi n}\left(\frac{n-1}{e}\right)^{n-1}=(1+o(1)) (n-1)!\,.
\end{multline*}
In view of the innermost integration limits in \eqref{Int1}, $\sigma_i$ are given by
\[
\sigma_1(\boldsymbol\xi)=\Biggl(n^{\a}\prod_{s\in [r-1]}\xi^{(s)}\Biggr)^{1/r},\quad \sigma_2(\boldsymbol\xi)=\Biggl(n^{\be}\prod_{s\in [r-1]}\xi^{(s)}\Biggr)^{1/r}.
\]
The conditions \eqref{sigmas} will easily hold if
\begin{equation}\label{easy}
2n^{1-\be}\le \prod_{s=1}^{r-1}\xi^{(s)}\le 0.5 n^{1-\a}.
\end{equation}
So, uniformly for $\boldsymbol\xi$ meeting this condition, the innermost integral is asymptotic to $(n-1)!$. It remains  to estimate the integral obtained
from the one in \eqref{Int1} by replacing the innermost integral with $(n-1)!$, and adding the constraint \eqref{easy} to the integration range of $\boldsymbol\xi$. Introducing the new variables $y^{(s)}=\log \xi^{(s)}/(\be\log n)$, $(s\in [r-1])$, we obtain an asymptotic lower bound for the integral in \eqref{Int1}:
\begin{align*}
&(n-1)! (\log n)^{r-1}\be^{r-1}\!\!\!\!\!\idotsint\limits_{\bold y\in [\a/\be, 1]^{r-1}\atop (1-\be)/\be \le \sum_s y^{(s)}\le (1-\a)/\be}\!\!\!\!\!\!\! 1\cdot \prod_{s\in [r-1]} dy^{(s)}\\
&=(n-1)! (\log n)^{r-1}\be^{r-1} \Bbb P (T_{r-1}\in [(1-\be)/\be, (1-\a)/\be]);
\end{align*}
here $T_{r-1}=\sum_{s\in [r-1]}Y^{(s)}$, and $Y{(s)}$ are independent, $[0,1]$-Uniforms. Pushing $\a$ toward $0$ and $\be$ toward $1/2$, we can make 
this estimate arbitrarily close to 
\begin{equation}\label{lower1}
(n-1)! \left(\frac{\log n}{2}\right)^{r-1}\!\! \Bbb P (T_{r-1}\in [1, 2]).
\end{equation}
{\bf (2)\/} Turn to the second integral in \eqref{newP(M)>prod}. Observe that 
\[
(n-1)!\cdot \Bbb I\Biggl(\,\sum_{\nu\in [n-1]}\ell_{\nu}\le 1\!\Biggr)
\]
is the joint density of the lengths $L_1,\dots,L_{n-1}$ of the first $n-1$ intervals obtained by throwing uniformly at random $n-1$ points into the interval $[0,1]$. Therefore, by the definition of constraints \eqref{2Omnew}, \eqref{3Omnew} we have 
\begin{multline}\label{lower2}
\idotsint\limits_{\{\bold u^{(s)}\}_{s\in [r-1]}\text{ meet }\eqref{2Omnew}, \eqref{3Omnew}} 1\cdot \prod_{s\in [r-1]}d\bold u^{(s)}\\
=\frac{1}{\bigl((n-1)!\bigr)^r}\,\,\Bbb P^r\!\left(\max_{\nu \in [n]} L_{\nu}\le \frac{2\log n}{n},\quad \sum_{\nu\in [n]}L_{\nu}^2\le \frac{3}{n}\right)\\
=(1+o(1))\frac{1}{\bigl((n-1)!\bigr)^r},
\end{multline}
for the last equality see Pittel \cite{Pit1}.

Multiplying the lower bounds \eqref{lower1} and \eqref{lower2} we obtain
\[
\Bbb P(M)\ge (1+o(1)\left(\frac{\log n}{2 (n-1)!}\right)^{r-1}\Bbb P (T_{r-1}\in [1, 2]).
\]
Finally
\[
\Bbb E[S_n] =(n!)^{r-1}\Bbb P(M)\ge (1+o(1))\left(\frac{n\log n}{2}\right)^{r-1}\Bbb P (T_{r-1}\in [1, 2]).
\]
\end{proof}

\noindent {\bf Note.\/} 
The condition $r\ge 3$ played an important role in the  proof,
making $\phi^{''}(\sigma^*)=-\Theta(n^{\frac{r-2}{r}})\to-\infty$, and enabling us to claim that the dominant contribution to the innermost integral in \eqref{Int1} came from $\sigma$s within a factor $1+o(1)$ from $\sigma^*$. For $r=2$ this is not true, since $\phi^{''}(\sigma^*)=O(1)$. 
\section{Two-side stable matchings with partially ordered preferences.} We have two agent sets, $\mathcal A_1$ and $\mathcal A_2$, which are copies of $[n]$.
Suppose that the
preference lists of $n$ agents from $\mathcal A_1$ (from $\mathcal A_2$ resp.) for a marriage partner in $\mathcal A_2$ (in $\mathcal A_1$ resp.) are $n$ independent copies of $\mathcal P_{k_1}(n)$ ($\mathcal P_{k_2}(n)$ resp.),  where $\mathcal P_k(n)$ is the uniformly random $k$-dimensional
partial order on $[n]$, Winkler \cite{Win1}. (In addition, the $n$ copies of $\mathcal P_{k_1}(n)$ are independent of the $n$ copies of $\mathcal P_{k_2}(n)$.)

A generic $\mathcal P_k(n)$ is constructed by taking $k$ linear orders on the set $[n]$, uniformly and independently at random, from the set of all $n!$ linear orders, and forming the intersection of $k$ orders. Equivalently, $\mathcal P_k(n)$ can be generated
by throwing $n$ points $\bold Z^{(1)},\dots, \bold Z^{(n)}$, uniformly and independently of each other, into the cube $[0,1]^k$, equipped with the coordinate (linear) order
$\preceq$, and taking the order $\mathcal P$ on the $n$ points induced by $\preceq$. Uniformity of $\bold Z^{(j)}$ means that its coordinates 
$Z^{(j)}_1,\dots, Z^{(j)}_k$ are independent $[0,1]$-Uniforms. The $k$ coordinate orders $\mathcal P_1,\dots, \mathcal P_k$ are independent of each other, and $\mathcal P=\cap_{j=1}^k\mathcal P_j$, whence $\mathcal P=\mathcal P_k(n)$. Neglecting a zero-probability event, we have
\begin{equation}\label{negl}
\{\bold Z^{(\a)}\preceq\bold Z^{(\be)}\}=\{\bold Z^{(\a)} \boldsymbol{<}\bold Z^{(\be)}\}:=\bigcap_{u\in [k]}\{Z^{(\a)}_u<Z^{(\be)}_u\}.
\end{equation}

Back to $\mathcal A_1$ and $\mathcal A_2$, we introduce  $\bold X_i=\{\bold X_i^{(j)}\}_{j\in \mathcal A_2}$, ($i\in\mathcal A_1$), and 
$\bold Y_j=\{\bold Y_j^{(i)}\}_{i\in \mathcal A_1}$, ($j\in\mathcal A_2$), such that all $\bold X_i^{(j)}$ and $\bold Y_j^{(i)}$ are independent, and each 
$\bold X_i^{(j)}$ is $[0,1]^{k_1}$-uniform, while each of $\bold Y_j^{(i)}$ is $[0,1]^{k_2}$-uniform. For each $i\in \mathcal A_1$ ($j\in \mathcal A_2$ resp.)
the $n$-long sequence $\bold X_i$ ($\bold Y_j$ resp.) induces the random partial order $\mathcal P_{i,k_1}(n)$ on $\mathcal A_2$ ($\mathcal P_{j,k_2}(n)$ on $\mathcal A_1$ resp.), with all the partial orders independent of each other. That's our stable matching problem with random partially ordered preference lists.

Let $M$ be the bijective mapping from $\mathcal A_1$ to $\mathcal A_2$, such that, {\it numerically\/} $M(i)=i$, ($i\in [n]$), so that $M^{-1}(j)=j$, ($j\in [n]$).

$M$ is weakly stable (w-stable)
if no unmatched pair $i\in \mathcal A_1,\,j\in \mathcal A_2$ is such that $i$ strictly prefers $j$ to its partner $M(i)$ and $j$ strictly prefers $i$ to its partner $M^{-1}(j)$. So by  \eqref{negl}, we have
\[
\{M\,\text{w-stable}\}=\!\!\bigcap_{i\in \mathcal A_1,\, j\in\mathcal A_2\atop i\neq j}\Bigl\{\bold X_i^{(j)}<\bold X_i^{(i)},\,
\bold Y_j^{(i)}<\bold Y_j^{(j)}\Bigr\}^c.
\]
Conditioned on the event
\[
\{\bold X_{\a}^{(\a)}=\bold x_{\a},\bold Y_{\be}^{(\be)}=\bold y_{\be}\}_{\a\in \mathcal A_1,\be\in\mathcal A_2},  \,\, \bold x_{\a}\!=\!\{x_{\a,u}\}\!\in [0,1]^{k_1},
\,\bold y_{\be}\!=\!\{y_{\be,v}\}\!\in [0,1]^{k_2},
\]
 the events in the above intersection are independent. Therefore, using $|\circ$ to denote the conditioning, we have
\begin{multline}\label{w-stable}
\Bbb P(M\,\text{w-stable}|\circ)\\=\prod_{i\in \mathcal A_1,\, j\in\mathcal A_2\atop i\neq j}\left(1-\Bbb P\bigl(\bold X_i^{(j)}<\bold X_i^{(i)},\,
\bold Y_j^{(i)}<\bold Y_j^{(j)}|\circ\bigr)\right)\\
=\prod_{1\le i\neq j\le n}\Biggl(1-\prod_{u\in [k_1]} x_{i,u}\cdot \prod_{v\in [k_2]} y_{j,v}\Biggr).
\end{multline}
\indent Next, $M$ is strongly stable (s-stable) if no unmatched pair $i\in \mathcal A_1,\,j\in \mathcal A_2$ is such that {\it either\/} $i$ strictly prefers $j$ to its partner $M(i)$ and $j$ does not strictly prefer its partner $M^{-1}(j)$ to $i$, {\it or\/} $j$ strictly prefers $i$ to its partner $M^{-1}(j)$ and $i$ does not strictly prefer its partner $M(i)$ to $j$. Therefore
\begin{multline*}
\{M\,\text{s-stable}\}\\
=\!\!\bigcap_{i\in \mathcal A_1,\, j\in\mathcal A_2\atop i\neq j}\Biggl(\!\Bigl\{\bold X_i^{(j)}<\bold X_i^{(i)},\,
\bold Y_j^{(i)}\not > \bold Y_j^{(j)}\Bigr\}\bigcup
 \Bigl\{\bold X_i^{(j)}\not >\bold X_i^{(i)},\,\bold Y_j^{(i)}<\bold Y_j^{(j)}\Bigr\}\!\Biggr)^c.\\
 \end{multline*}
Therefore
\begin{multline}\label{s-stable}
\Bbb P(M\,\text{s-stable}|\circ)
=\!\!\!\!\prod_{1\le i\neq j\le n}\Biggl[1-\prod_{u\in [k_1]}x_{i,u}\cdot\!\Biggl(\!1-\prod_{v\in [k_2]}(1-y_{j,v})\!\Biggr)\\
-\prod_{v\in [k_2]} y_{j,v}\cdot\!\Biggl(\!1-\!\prod_{u\in [k_1]}(1-x_{i,u})\!\Biggr)
+\prod_{u\in [k_1]}x_{i,u}\cdot \prod_{v\in [k_2]} y_{j,v}\Biggr].
\end{multline}
\indent Finally, $M$ is super-stable (sup-stable) if no unmatched pair $i\in \mathcal A_1,\,j\in \mathcal A_2$ is such that $i$ does not strictly prefer $M(i)$ to $j$ and $j$ does not
strictly prefer $M^{-1}(j)$ to $i$. Therefore 
\[
\{M\,\text{sup-stable}\}=\!\!\bigcap_{i\in \mathcal A_1,\, j\in\mathcal A_2\atop i\neq j}\Bigl\{\bold X_i^{(i)}\not<\bold X_i^{(j)},\,
\bold Y_j^{(j)}\not<\bold Y_j^{(i)}\Bigr\}^c,
\]
so that
\begin{multline}\label{sup-stable}
\Bbb P(M\,\text{sup-stable}|\circ)=\prod_{1\le i\neq j\le n}\Bigl[1-\Bbb P\bigl(\bold X_i^{(i)}\not < \bold X_i^{(j)}|\circ\bigr)\,\Bbb P\bigl(\bold Y_j^{(j)}\not <\bold Y_j^{(i)}|\circ\bigr)\Bigr]\\
=\!\!\!\!\prod_{1\le i\neq j\le n}\Biggl[1-\Biggl(1-\prod_{u\in [k_1]}(1-x_{i,u})\Biggr)\Biggl(1-\prod_{v\in [k_2]}(1-y_{j,v})\Biggr)\Biggr].
\end{multline}

Let $\bold x=\{\bold x_{\a}\}_{\a\in \mathcal A_1}$, $\bold y=\{\bold y_{\be}\}_{\be\in \mathcal A_2}$, so that $\bold x\in [0,1]^{nk_1}$, $y\in [0,1]^{nk_2}$.
Unconditioning \eqref{w-stable}, \eqref{s-stable} and \eqref{sup-stable}, we have proved
\begin{lemma}\label{P(Mstab)=}
\begin{align*}
\Bbb P_{k_1,k_2}(M\,\text{w-stable})&=\idotsint\limits_{\bold x\in [0,1]^{nk_1},\, y\in [0,1]^{nk_2}} F_1(\bold x,\bold y)\,d\bold x\,d\bold y,\\
\Bbb P_{k_1,k_2}(M\,\text{s-stable})&=\idotsint\limits_{\bold x\in [0,1]^{nk_1},\, y\in [0,1]^{nk_2}} F_2(\bold x,\bold y)\,d\bold x\,d\bold y,\\
\Bbb P_{k_1,k_2}(M\,\text{sup-stable})&=\idotsint\limits_{\bold x\in [0,1]^{nk_1},\, y\in [0,1]^{nk_2}} F_3(\bold x,\bold y)\,d\bold x\,d\bold y,
\end{align*}
with $F_1$, $F_2$ and $F_3$  being the RHS expressions in \eqref{w-stable}, \eqref{s-stable} and \eqref{sup-stable}. For $k_1=k_2=1$, all three functions collapse into $F(\bold x,\bold y)=\prod_{1\le i\neq j\le n} (1-x_iy_j)$. In addition, $F_2(\bold x,\bold y)=F_3(\bold x,\bold y)$ for $k_2=1$.
This is not surprising since for $k_2=1$ super-stability is the same as strong stability.
\end{lemma}
By symmetry, these probabilities do not depend on the choice of a matching $M$. Let $S_{n,w}$, $S_{n,s}$ and $S_{n,sup}$ denote the total number of
weakly stable, strongly stable and super stable matchings. Then Lemma \ref{P(Mstab)=} implies
\begin{equation}\label{ESnstab=}
\begin{aligned}
\Bbb E[S_{n,w}]&=n! \idotsint\limits_{\bold x\in [0,1]^{nk_1},\, \bold y\in [0,1]^{nk_2}} F_1(\bold x,\bold y)\,d\bold x\,d\bold y,\\
\Bbb E[S_{n,s}]&=n! \idotsint\limits_{\bold x\in [0,1]^{nk_1},\, \bold y\in [0,1]^{nk_2}} F_2(\bold x,\bold y)\,d\bold x\,d\bold y,\\
\Bbb E[S_{n,sup}]&=n! \idotsint\limits_{\bold x\in [0,1]^{nk_1},\, \bold y\in [0,1]^{nk_2}} F_3(\bold x,\bold y)\,d\bold x\,d\bold y.\\
\end{aligned}
\end{equation}
In \cite{Pit1}  for $k_1=k_2=1$, i.e. for the random totally ordered preference lists, we proved that $\Bbb E[S_n]\sim e^{-1} n\log n$. We will
analyze asymptotics of the three expectations for $\{k_1, k_2\}\neq\{1,1\}$. 

Observe that the two random orders on $[n]$, of dimension $k_1$ and $k_1+1$, can be naturally coupled in such a way that the latter 
 is the intersection of the former and the random independent total order.  Weak stability of a given matching $M$ means that there are no
 destabilizing unmatched pairs, i.e. strictly preferring each other to their partners in the matching. Under the coupling, the set of pairs destabilizing
$M$ for the $(k_1+1,k_2)$-problem is contained in the set of pairs destabilizing $M$ for the $(k_1,k_2)$-problem. It follows that $\Bbb P_{k_1,k_2}(M\,\text{w-stable})\le \Bbb P_{k_1+1,k_2}(M\,\text{w-stable})$, i.e.  $\Bbb P_{k_1,k_2}(M\,\text{w-stable})$ is an increasing function of $k_1,\,k_2$.

On the other hand, super-stability of $M$ means absence of a radically less selective set of unmatched pairs: a pair $(i,j)$ is classified as destabilizing
if neither $i$ strictly prefer $M(i)$ to $j$, nor $j$ strictly prefers $M^{-1}(j)$ to $i$. Under the coupling, this set is increasing with $k_1$ and $k_2$.
Therefore $\Bbb P_{k_1,k_2}(M\,\text{sup-stable})$ is a decreasing function of $k_1,\,k_2$.

\subsection{Super-stable matchings.} First, let us upper-bound $F_3(\bold x,\bold y)$ defined in \eqref{sup-stable}. Using 
the geometric-arithmetic mean inequality, we have
\begin{equation}\label{crudeprod}
\prod_{u\in [k_1]}(1-x_{i,u})\le \left(1-\frac{1}{k_1}\sum_{u\in [k_1]}x_{i,u}\right)^{k_1}\le 1-\frac{1}{k_1}\sum_{u\in [k_1]}x_{i,u},
\end{equation}
and likewise 
\[
\prod_{v\in [k_2]}(1-y_{j,v})\le1-\frac{1}{k_2}\sum_{v\in [k_2]}y_{j,v}.
\]
\begin{equation}\label{F_3<}
F_3(\bold x,\bold y)\le \prod_{1\le i\neq j\le n}\Biggl[1-\frac{1}{k_1k_2}\Biggl(\sum_{u\in [k_1]}x_{i,u}\Biggr)\cdot \Biggl(\sum_{v\in [k_2]}y_{j,v}\Biggr)\Biggr].
\end{equation}
Let $f_k(z)$ denote the density of $\sum_{w\in [k]}Z_w$, $Z_w$ being independent $[0,1]$-Uniforms. 
Introducing $\boldsymbol\xi=\{\xi_{\a}\}_{\a\in \mathcal A_1}$, $\boldsymbol\eta=\{\eta_{\be}\}_{\be\in \mathcal A_2}$, $\boldsymbol\xi\in [0,k_1]^n$,
$\boldsymbol\eta\in [0,k_2]^n$, we obtain from \eqref{ESnstab=} that
\begin{align}
\Bbb E[S_{n,sup}]&\le n!\, I_n,\quad I_n:=\idotsint\limits_{\boldsymbol \xi\in [0,k_1]^n,\, \boldsymbol\eta\in [0,k_2]^n} F_3^*(\boldsymbol \xi,\boldsymbol\eta)\,d\boldsymbol\xi\,d\boldsymbol \eta,\label{In}\\
F_3^*(\boldsymbol \xi,\boldsymbol\eta)&=\prod_{1\le i\neq j\le n}\!\Biggl(1-\frac{\xi_i\eta_j}{k_1k_2}\Biggr)\cdot\prod_{i\in [n]}f_{k_1}(\xi_i)\cdot\prod_{j\in [n]}f_{k_2}(\eta_j).\label{F*}
\end{align}
\indent Suppose $k_1>1$. Using $1-\zeta\le e^{-\zeta}$ and $f_{k_2}(\eta)\le \eta^{k_2-1}/(k_2-1)!$, we have
\begin{equation}\label{supp}
\begin{aligned}
&I_n\le \idotsint\limits_{\boldsymbol\xi\in [0,k_1]^n}\Biggl(\prod_{j=1}^n\int_0^{k_2}\exp\Biggl(-\frac{\eta s_j}{k_1k_2}\Biggr)\frac{\eta^{k_2-1}}{(k_2-1)!}\,d\eta\Biggr)
\prod_{i\in [n]}f_{k_1}(\xi_i)\,d\boldsymbol\xi,\\
&\qquad\qquad\qquad\qquad\qquad\quad  s_j:=\sum_{i\neq j}\xi_i. 
\end{aligned}
\end{equation}
\indent Fix $a\in (0,1)$ and write $I_n=I_{n,1}+I_{n,2}$ where $I_{n,1}$ is the contribution of $\boldsymbol\xi$ with $s:=\sum_{i\in [n]}\xi_i\le n^a$,
and $I_{n,2}$ is the contribution of $\boldsymbol\xi$ with $s>n^a$. Introducing $\ga:=k_2^{k_2-1}/(k_2-1)!$ and integrating $e^{-\eta s_j/k_1k_2}$,  we obtain
\[
I_{n,1}\le\ga_1^n \idotsint\limits_{s\le n^a}\prod_{j\in [n]}\frac{1-\exp\bigl(-s_j/k_1\bigr)}{s_j/k_1}\prod_{i\in [n]}f_{k_1}(\xi_i)\,d\boldsymbol\xi,
\quad \ga_1=k_2\ga.
\]
Now
\begin{multline}\label{elem}
\Biggl(\log\frac{1-e^{-z}}{z}\Biggr)'=-\frac{e^z-1-z}{z(e^z-1)}\\
=-\frac{\sum_{j\ge 2}z^j/j!}{\sum_{j\ge 2}z^j/(j-1)!}\in [-\max_{j\ge 2}1/j,0]=[-1/2,0].
\end{multline}
Using this inequality and $s-s_j=\xi_j$, we have
\begin{equation}\label{elem1}
\log\frac{1-\exp\bigl(-s_j/k_1\bigr)}{s_j/k_1}\le \frac{1-\exp\bigl(-s/k_1\bigr)}{s/k_1} +\frac{1}{2k_1}\, \xi_j,
\end{equation}
implying that
\begin{equation}\label{elem2}
\sum_{j\in [n]}\log\frac{1-\exp\bigl(-s_j/k_1\bigr)}{s_j/k_1}\le n\log\frac{1-\exp\bigl(-s/k_1)\bigr)}{s/k_1} +\frac{1}{2k_1} s.
\end{equation}
Therefore
\begin{multline*}
I_{n,1}\le \ga_1^n \idotsint\limits_{s\le n^a}\left(\frac{1-e^{-s/k_1}}{s/k_1}\right)^n e^{s/2k_1}
\prod_{i\in [n]}f_{k_1}(\xi_i)\,d\boldsymbol\xi\\
=\ga_1^n\cdot\Bbb E\left[\left(\frac{1-e^{-\mathcal S_n/k_1}}{\mathcal S_n/k_1}\right)^n\cdot e^{\mathcal S_n/2k_1}\cdot\Bbb I(\mathcal S_n\le n^a)\right],
\end{multline*}
where $\mathcal S_n$ is the sum of $k_1n$ independent $[0,1]$-Uniforms. So, since the density of $\mathcal S_n$ is bounded by $s^{k_1n-1}/(k_1n-1)!$, we drop $(1-e^{-\mathcal S_n/k_1})^n$  and obtain
\begin{align*}
I_{n,1}&\le \ga_1^n k_1^n e^{cn^a/k_1}\int_0^{n^a}\frac{s^{(k_1-1)n-1}}{(k_1n-1)!}\,ds\\
&\le (\ga_1k_1+o(1))^n\cdot\frac{n^{a(k_1-1)n}}{(k_1n)!} \le \ga_2^n\cdot n^{n[a(k_1-1)-k_1]},\quad (\ga_2:=3\ga_1),
\end{align*}
as $(k_1n)!> (k_1n/e)^{k_1n}$. Therefore
\begin{equation}\label{In1<}
n!\cdot I_{n,1}\le \gamma_2^n\cdot n^{n(k_1-1)(a-1)}\to 0,
\end{equation}
since $k_1>1$ and $a<1$.

\indent Turn to $I_{n,2}$. We estimate
\begin{multline*}
\int_0^{k_2} e^{-\frac{\eta s_j}{k_1k_2}}\,\frac{\eta^{k_2-1}}{(k_2-1)!}\,d\eta
\le \left(\frac{k_1k_2}{s_j}\right)^{k_2}\int_0^{\infty} e^{-z}\frac{z^{k_2-1}}{(k_2-1)!}\,dz =\left(\frac{k_1k_2}{s_j}\right)^{k_2},
\end{multline*}
implying that
\[
I_{n,2}\le \idotsint\limits_{s>n^a}\prod_{j\in [n]}\left(\frac{k_1k_2}{s_j}\right)^{k_2}\prod_{i\in [n]}f_{k_1}(\xi_i)\,d\boldsymbol\xi.
\]
Since $s\ge n^a$, we have 
\begin{equation}\label{prodsj}
\prod_{j\in [n]}s_j^{-1}=s^{-n}\exp\bigl(1+O(n^{-a})\bigr).
\end{equation}
Therefore, picking $\ga_3> (k_1k_2)^{k_2}$, and $\ga_4>\ga_3 (e^{k_1}/k_1k_2)$,  we have: for $n$ large enough,
\begin{align*}
I_{n,2}&\le\ga_3^n\idotsint\limits_{s>n^a} s^{-k_2n}\prod_{i\in [n]}f_{k_1}(\xi_i)\,d\boldsymbol\xi\\
&\le\ga_3^n\int_{n^{\a}}^{k_1n}\frac{s^{(k_1-k_2)n-1}}{(k_1n-1)!}\,ds\le \ga_4^n\,n^{-k_2n}.
\end{align*}
Consequently
\begin{equation}\label{In2<}
n!\cdot I_{n,2}\le \ga_4^n\, n^{-(k_2-1)n}.
\end{equation}
Combining \eqref{In1<} and \eqref{In2<}, we arrive at 
\begin{lemma}\label{nosup} Suppose that $k_1,\, k_2\ge 2$. Then
\[
\Bbb E[\mathcal S_{n,\text{sup}}]\le n^{-n\bigl[\min(k_1-1,k_2-1)-o(1)\bigr]},
\]
implying (by Markov inequality) that
\[
\Bbb P(S_{n,\text{sup}} >0) \le n^{-n\bigl[\min(k_1-1,k_2-1)-o(1)\bigr]}.
\]
In words, the fraction of problem instances with at least one super-stable matching is super-exponentially small.
\end{lemma}

Consider the remaining case $k_1\ge 2$ and $k_2=1$. Here $\{y_{j,v}\}=y_{j,1}=:y_j$, and  so
\begin{equation*}
F_3(\bold x,\bold y)=\prod_{1\le i\neq j\le n}\Biggl[1-y_j\Biggl(1-\prod_{u\in [k_1]}(1-x_{i,u})\Biggr)\Biggl].
\end{equation*}
$\prod_{u\in [k_1]}(1-x_{i,u})$ can be viewed as the generic value of $\prod_{i\in [k_1]} (1-X_{i,u})$, $X_{i,u}$ being
independent Uniforms. Obviously $1-X_{i,u}$ are also independent Uniforms. It is known that the product of $k$ independent $[0,1]$-Uniforms has density $\phi_k(z):=[\log^{k-1}(1/z)]/(k-1)!$.  (A simple inductive proof is based on a recurrence
\[
f_k(z)=\int_z^1 \eta^{-1} f_{k-1}(z/\eta)\,d\eta. )
\]
Then $1$ minus the random product has
density $\psi_k(z)=\frac{\log^{k-1} (1-z)^{-1}}{(k-1)!}$, $(z\in (0,1])$, and we are back to the uniform density if $k=1$. So, introducing the sequence $\{Z_i\}_{i\in [n]}$ of independent random variables with 
common density $\psi_{k_1}(z)$ we obtain that
\begin{equation}\label{E(YZ)}
\Bbb E[S_{n,\text{sup}}]=n!\cdot I_n, \quad I_n:=\Bbb E\Biggl[\prod_{1\le i\neq j\le n}\Bigl(1-Y_j Z_i\Bigr)\Biggr];
\end{equation}
here $Y_j$ are $[0,1]$-Uniforms, which are independent among themselves and from $\{Z_i\}$. Analogously to \eqref{supp}, we write
\begin{equation*}
I_n\le \idotsint\limits_{\boldsymbol z\in [0,1]^n}\Biggl(\prod_{j=1}^n\int_0^1 e^{-y s_j}\,dy\Biggr)
\prod_{i\in [n]}\psi_{k_1}(z_i)\,d\boldsymbol z,\quad s_j:=\sum_{i\neq j} z_i.
\end{equation*}
The innermost integral is $(1-e^{-s_j})/s_j$; so arguing as in \eqref{elem}-\eqref{elem2}, we obtain
\[
I_n\le \idotsint\limits_{\bold z\in [0,1]^n}e^{s/2}\,\left(\frac{1-e^{-s}}{s}\right)^n \prod_{i\in [n]}\psi_{k_1}(z_i)\,d\bold z,\quad s:=\sum_{i\in [n]}z_i.
\]
Unlike the case of the uniform density, we have no tractable  {\it upper\/} bound for the $n$-th order convolution of the density $\psi_{k_1}(z)$
with itself. Fortunately
$(1-e^{-s})/s$ is log-convex:
\[
\Bigl(\log(1-e^{-s})-\log s\Bigr)^{\prime\prime}=s^{-2}-(e^{s/2}-e^{-s/2})^{-2}>0.
\]
(Alternatively, this function is the Laplace transform of the uniform density, and it is known that Laplace transform of any nonnegative function
is log-convex.) Therefore
\[
\left(\frac{1-e^{-s}}{s}\right)^n\le \prod_{i\in [n]}\frac{1-e^{-nz_i}}{nz_i}\le \prod_{i\in [n]}\frac{1}{nz_i},
\]
so that
\[
I_n\le n^{-n} J_{k_1}^n,\quad J_{k_1}:=\int_0^1 e^{z/2}\,z^{-1}\psi_{k_1}(z)\,dz.
\]
Consequently
\[
\Bbb E[S_{n,\text{sup}}]\le n!\,I_n =O(n^{1/2}\rho_{k_1}^n),\quad \rho_{k}= e^{-1}J_{k}.
\]
By Maple: $\rho_2=0.8287956957,\,\rho_3=0.8287956957,\,\rho_4=0.6329102250$. In general, $\rho_k$ decreases with
$k$ increasing, because $1-\prod_{u\in [k]}X_u\le 1-\prod_{i\in [k+1]}X_u$, and $z^{-1} e^{z/2}$ is decreasing
for $z\in [0,1]$.
Therefore we have proved

\begin{lemma}\label{yessup} Suppose $\min(k_1,k_2)=1$ and $k:=\max(k_1,k_2)>1$. Then  for $n$ large enough
we have $\Bbb E[S_{n,\text{sup}}]\le (\rho_k+o(1))^n$, $\rho_k$ decreases with $k$ and $\rho_2<0.83$.
\end{lemma}

A sharp contrast between the bounds in Lemma \ref{nosup} and Lemma \ref{yessup} raises the question: is  $\Bbb E[S_{n,\text{sup}}]$ {\it exactly\/} exponentially small under the conditions of Lemma \ref{yessup}? To answer positively, we need
a sufficiently sharp lower bound for $I_n$ in \eqref{E(YZ)}. We start with a bound
\begin{align*}
&\qquad\qquad\qquad I_n:=\Bbb E\Biggl[\prod_{1\le i\neq j\le n}\Bigl(1-Y_j Z_i\Bigr)\Biggr]\ge I_n^*,\\
&I_n^*:=\Bbb E\Bigl[\Bbb I\bigl(\max Y_j\le \be;\,\min Z_i\ge 1- e^{-1}\bigr)\cdot\prod_{1\le i\neq j\le n}\Bigl(1-Y_j Z_i\Bigr)\Bigr];
\end{align*}
here 
\[
\be=1-\frac{\hat{\a}}{\a},\quad \a>1\,\text{ and }\, \hat{\a}e^{-\hat{\a}}=\a e^{-\a},
\]
so that $\hat{\a}<1$. For this pair $(\a,\be)$, we have: if $Y_j\le\be$, then $1-Y_jZ_i\ge e^{-\a Y_jZ_i}$. 

\noindent (Indeed, for $\eta\le\be$, we have
\[
\a\ge \a (1-\eta)\ge \a(1-\be)=\hat{\a}.
\]
Therefore
\[
(1-\eta)e^{\a\eta}=\a(1-\eta)e^{-\a(1-\eta)}\cdot \a^{-1} e^{\a}\ge \hat{\a}e^{-\hat{\a}}\cdot \a^{-1}e^{\a}=1.)
\]

\noindent The
constraint $\min Z_i \ge 1-e^{-1}$ is imposed because the density $\psi_{k_1}(z)$ is log-convex for $z\in [1-e^{-1},1)$,
a property we use in the second line below.
Consequently
\begin{align*}
I_n^*&\ge \idotsint\limits_{\bold z\in [1-e^{-1},1]^n}\Biggl(\int_0^{\be} e^{-\a ys}\,dy\Biggr)^n\prod_{i\in [n]}\psi_{k_1}(z_i)\,d\bold z\\
&\ge \idotsint\limits_{\bold z\in [1-e^{-1},1]^n}\Biggl(\frac{1-e^{-\a\be s}}{\a s}\Biggr)^n\cdot \psi_{k_1}^n(s/n)\,d\bold z\\
&\ge\Biggl(\frac{1-e^{-\a\be n(1-e^{-1})}}{\a n}\Biggr)^n \bigl[(k_1-1)!\bigr]^{-n}\cdot  \idotsint\limits_{\bold z\in [1-e^{-1},1]^n}
1\, d\bold z\\
&\ge 0.5 \bigl(\a e (k_1-1)!\, n\bigr)^{-n},
\end{align*}
if $n$ is large. Therefore
\[
\Bbb E[S_{n,\text{sup}}]\ge n!\, I_n^*\ge 0.5\bigl(\a e^2(k_1-1)!\bigr)^{-n},
\]
for every $\a>1$ if $n\ge n(\a)$. Thus
\begin{lemma}\label{yessup!} Under the conditions of Lemma \ref{yessup},  we have
\[
\Bbb E[S_{n,\text{sup}}]\ge (r_k -o(1))^n,\quad r_k:=\bigl(e^2(k_1-1)!\bigr)^{-1}.
\]
\end{lemma}
In summary, we have proved
\begin{theorem}\label{no/yessup} If $\min(k_1,k_2)>1$, then the fraction of problem instances with at least one super-stable matching is at most
$n^{-n\bigl[\min(k_1-1,k_2-1)-o(1)\bigr]}$.
 If $\min(k_1,k_2)=1$ and $\max(k_1,k_2)>1$ then the  fraction of problem instances with at least one super-stable matching is between $(r_k-o(1))^n$ and $(\rho_k+o(1))^n$, where $0<r_k<\rho_k<1$ and $r_k,\,\rho_k$ decrease
 as $k$ increases.
\end{theorem}

\subsection{Strongly stable matchings} First of all,
\begin{lemma}\label{min=1}If $k:=\max(k_1,k_2)>1$ and $\min(k_1,k_2)=1$.  Then  for $n$ large enough we have
$
\Bbb E[S_{n,s}]\in [(r_k-o(1))^{n},\,0.83^n].
$
\end{lemma}
\noindent The proof is immediate since $F_2(\bold x,\bold y)=F_3(\bold x,\bold y)$ if $\min(k_1,k_2)=1$.

It remains to consider the case $\min(k_1,k_2)>1$. Let us upper-bound $\Bbb E[S_{n,s}]$. Recall (see \eqref{s-stable}) that
\begin{multline*}
F_2(\bold x,\bold y)
=\!\!\!\!\prod_{1\le i\neq j\le n}\Biggl[1-\prod_{u\in [k_1]}\!\!x_{i,u}\cdot\!\Biggl(\!1-\prod_{v\in [k_2]}(1-y_{j,v})\!\Biggr)\\
-\prod_{v\in [k_2]} \!\!y_{j,v}\cdot\!\Biggl(\!1-\!\prod_{u\in [k_1]}(1-x_{i,u})\!\Biggr)
+\prod_{u\in [k_1]}\!\!x_{i,u}\cdot \prod_{v\in [k_2]}\!\! y_{j,v}\Biggr].
\end{multline*}
We need to find a tractable upper bound for $F_2(\bold x,\bold y)$.
\begin{lemma}\label{prod-prod} If $z_{\ell}\in [0,1]$, ($\ell\in [k]$), then
\[
\prod_{\ell\in [k]}(1-z_{\ell})\le \Biggl[1-\Biggl(\prod_{\ell\in [k]} z_{\ell}\Biggr)^{\!\!1/k}\,\Biggr]^k\le 1-\Biggl(\prod_{\ell\in [k]} z_{\ell}\Biggr)^{\!\!1/k}.
\]
\end{lemma}
\begin{proof} If, given $p\in (0,1)$,
\[
\max\left\{\sum_{\ell} \log (1-z_{\ell}): z_{\ell}\in (0,1);\,\sum_{\ell\in [k]}\log z_{\ell}\ge \log p\right\}=k\log\bigl(1-p^{1/k}\bigr),
\] 
then the claim follows immediately. Since $\sum_{\ell} \log (1-z_{\ell})$ is concave and the range of $\bold z$ is convex, it suffices to produce $\la\ge 0$ such that $z^*_{\ell}\equiv p^{1/k}$
is a stationary point of the Lagrange function $\sum_{\ell} \log (1-z_{\ell}) +\la \sum_{\ell\in [k]}\log z_{\ell}$ in the open cube $(0,1)^n$. 
The needed $\la$ is $p^{1/k}/(1-p^{1/k})$.
\end{proof}
Using Lemma \ref{prod-prod}, we have 
\begin{multline*}
F_2(\bold x,\bold y)
\le \!\!\!\prod_{1\le i\neq j\le n}\Biggl[1-\prod_{u\in [k_1]}\!\!x_{i,u}\cdot\Biggl(\prod_{v\in [k_2]}\!\! y_{j,v}\Biggr)^{\!1/k_2}\\
-\prod_{v\in [k_2]}\!\! y_{j,v}\cdot\!\Biggl(\prod_{u\in [k_1]}\!\!x_{i,u}\Biggr)^{\!1/k_1}
+\prod_{u\in [k_1]}\!\!x_{i,u}\cdot \prod_{v\in [k_2]}\!\! y_{j,v}\Biggr].
\end{multline*}
The bound looks promising as it depends only on $2n$ products $P(\bold x_i)=\prod_{u\in [k_1]} x_{i,u}$, $P(\bold y_j)=\prod_{v\in [k_2]} y_{j,v}$. 
Observe that with $k:=\min(k_1,k_2)$
\begin{multline*}
P(\bold x_i) P(\bold y_j)^{1/k_2} +P(\bold x_i)^{1/k_1} P(\bold y_j)-P(\bold x_i) P(\bold y_j)\\
\ge P(\bold x_i) P(\bold y_j)^{1/k} +P(\bold x_i)^{1/k} P(\bold y_j)-P(\bold x_i) P(\bold y_j)\\
\ge 2\Bigl[P(\bold x_i) P(\bold y_j)\Bigr]^{\frac{k+1}{2k}}-P(\bold x_i) P(\bold y_j)\ge \Bigl[P(\bold x_i) P(\bold y_j)\Bigr]^{\frac{k+1}{2k}}.
\end{multline*}
Introduce independent random variables $X_1,\dots, X_n$ with density $\phi_{k_1}(z)=\frac{\log^{k_1-1}(1/z)}{(k_1-1)!}$, and $Y_1,\dots,Y_n$ with density $\phi_{k_2}(z)=\frac{\log^{k_2-1}(1/z)}{(k_2-1)!}$.
It follows that
\begin{equation}\label{X^aY^a}
\Bbb E[S_{n,s}]\le n!\,I_n,\quad I_n:=\Bbb E\Biggl[\prod_{1\le i\neq j\le n}\Bigl(1-X_i^{\a}\, Y_j^{\a}\Bigr)\Biggr],\quad \a:=\frac{k+1}{2k}.
\end{equation}
So, mimicking \eqref{supp}, we have
\begin{equation}\label{mim1}
\begin{aligned}
&I_n\le\idotsint\limits_{\boldsymbol x\in [0,1]^n}\Biggl(\prod_{j\in [n]}\int_0^1 e^{-s_jy^{\a}} \phi_{k_2}(y)\,dy\Biggr)\prod_{i\in [n]}\phi_{k_1}(x_i)\,\,d\boldsymbol x,\\
&\qquad\qquad\qquad\qquad s_j:=\sum_{i\neq j}x_i^{\a}.
\end{aligned}
\end{equation}
Let $s=\sum_{i\in [n]} x_i^{\a}$. Since
\[
\frac{d}{dz}\log\Biggl( \int_0^1e^{-zy^{\a}} \phi_{k_2}(y)\,dy\Biggr) =- \frac{\int_0^1 e^{-zy^{\a}}\, y^{\a}\,\phi_{k_2}(y)\,dy}{\int_0^1 e^{-zy^{\a}} \phi_{k_2}(y)\,dy}\in [-1,0],
\]
we obtain
\[
\int_0^1 e^{-s_jy^{\a}} \phi_{k_2}(y)\,dy\le \exp\bigl(x_j^{\a}\bigr)\int_0^1 e^{-sy^{\a}} \phi_{k_2}(y)\,dy,
\]
implying that
\begin{equation}\label{mim2}
\prod_{j\in [n]}\int_0^1 e^{-y^{\a}s_j} \phi_{k_2}(y)\,dy\le e^n \left(\int_0^1 e^{-y^{\a}s} \phi_{k_2}(y)\,dy\right)^{\!n}.
\end{equation}
Now $\int_0^1 e^{-sy^{\a}} \phi_{k}(y)\,dy$ is a slightly-disguised Laplace transform of a non-negative function, whence it is log-convex. Therefore
\begin{equation}\label{mim3}
\left(\int_0^1 e^{-sy^{\a}} \phi_{k}(y)\,dy\right)^{\!n}\le \prod_{i\in [n]}\int_0^1e^{-y^{\a}nx_i^{\a}}\phi_k(y)\,dy.
\end{equation}
Combining \eqref{mim1}-\eqref{mim2}, we obtain
\begin{equation}\label{I_n<int}
\begin{aligned}
I_n&\le e^n\left(\int_0^1\int_0^1e^{-n(xy)^{\a}}\phi_{k_1}(x)\phi_{k_2}(y)\,dxdy\right)^n\\
&=e^n\left(\int_0^1e^{-nz^{\a}}\phi_{k_1+k_2}(z)\,dz\right)^n,\quad \a=\frac{k+1}{2k}.
\end{aligned}
\end{equation}
Indeed, $\phi_{k_1}(x)\phi_{k_2}(y)$ is the joint density of $\mathcal X,\,\mathcal Y$, the product of $k$ independent Uniforms and the product of another $k_2$ independent Uniforms, respectively. So the double integral is $\Bbb E[e^{-n(\mathcal X\mathcal Y)^{\a}}]=\Bbb E[e^{-nZ^{\a}}]$, where $Z$ is the product of $k_1+k_2$ independent Uniforms.
Substituting $nz^{\a}=\eta$, we obtain that the integral is asymptotic to
\[
n^{-1/a}\frac{(\a^{-1}\log n)^{k_1+k_2-1}}{(k_1+k_2-1)!}\cdot \frac{1}{\a}\int_0^{\infty}e^{-\eta}\,\eta^{1/\a-1}\,d\eta= c_1 n^{-1/\a} (\log n)^{k_1+k_2-1}.
\]
Recalling that $\a=(k+1)/2k$,  it follows that
\[
\Bbb E[S_{n,s}]\le n! e^n \Bigl[c_1n^{-1/a}(\log n)^{k_1+k_2-1}\Bigr]^n\!\! =O\Bigl(n^{1/2} (\log n)^{k_1+k_2-1} c_1^n n^{-\frac{k-1}{k+1}n}\Bigr).
\]
We have proved
\begin{lemma}\label{min>1} If $k:=\min(k_1,k_2)>1$, then
$
\Bbb E[S_{n,s}]\le n^{-n\bigl(\frac{k-1}{k+1}-o(1)\bigr)}.
$
\end{lemma}

In summary, we have
\begin{theorem}\label{strong} If $k:=\min(k_1,k_2)>1$, then the fraction of problem instances with at least one strongly stable matching is 
at most $n^{-n\bigl[\frac{k-1}{k+1}-o(1)\bigr]}$. If $\min(k_1,k_2)=1$ and $k_{\text{max}}:=\max(k_1,k_2)>1$ then the fraction of problem instances with at least one strongly stable matching is at most $0.83^n$, and at least 
$[e^2(k_{\text{max}}-1)!]^{-n}$.
\end{theorem}
{\bf Note.\/} For $k=\min(k_1,k_2)>1$ the fractions of solvable problem instances are super-exponentially small for both super-stable solutions and strongly
stable solutions. The difference is that for the former this fraction is around $n^{-n(k-1)}$, while for the latter the still
minuscule fraction is much larger, around
$n^{-n\frac{k-1}{k+1}}$.

\subsection{Weakly stable matchings.}
According to \eqref{ESnstab=} and \eqref{w-stable}, we have
\begin{equation}\label{start}
\begin{aligned}
&\,\,\Bbb E[S_{n,w}]=n! \idotsint\limits_{\bold x\in [0,1]^{nk_1},\, \bold y\in [0,1]^{nk_2}}\!\!\!\ \!\!\!\!\!F_1(\bold x,\bold y)\,d\bold x\,d\bold y,\\
&F_1(\bold x,\bold y):=\prod_{1\le i\neq j\le n}\!\!\Biggl(1-\prod_{u\in [k_1]}\!\! x_{i,u}\cdot \prod_{v\in [k_2]}\!\! y_{j,v}\Biggr),
\end{aligned}
\end{equation}
with $S_{n,w}$ denoting the total number of weakly stable matchings. This time $\Bbb E[S_{n,w}]\ge \Bbb E[S_n]$, which is the expected
number of stable matchings with random totally ordered preference lists, and it is asymptotic to $e^{-1}n\log n$.  
$\Bbb E[S_{n,w}]$ is expected to grow faster with $n$, but how much faster?  

{\bf (i)\/} Upper bound. Instead of the ``hard-won'' inequality \eqref{X^aY^a}, now, according to \eqref{start}-\eqref{mim1}, we have the analogous {\it equality\/} from the start:
\begin{equation}\label{XY}
\begin{aligned}
&\quad\,\,\Bbb E[S_{n,w}]= n!\,I_n,\quad I_n:=\Bbb E\Biggl[\prod_{1\le i\neq j\le n}\Bigl(1-X_i\, Y_j\Bigr)\Biggr],\\
&I_n\le\idotsint\limits_{\boldsymbol x\in [0,1]^n}\Biggl(\prod_{j\in [n]}\int_0^1 e^{-s_jy} \phi_{k_2}(y)\,dy\Biggr)\prod_{i\in [n]}\phi_{k_1}(x_i)\,\,d\boldsymbol x,\\
&\qquad\qquad\quad\qquad\,\, s_j:=\sum_{i\neq j}x_i.
\end{aligned}
\end{equation}
Following the steps that led us to \eqref{I_n<int}, we obtain
\begin{align*}
&\qquad\qquad\qquad I_n\le e^n K_n^n,\\
K_n&:= \int_0^1e^{-nz}\phi_{k_1+k_2}(z)\,dz\sim n^{-1} \frac{(\log n)^{k_1+k_2-1}}{(k_1+k_2-1)!}.
\end{align*}
We conclude that
\begin{equation}\label{wupper}
\Bbb E[S_{n,w}]\le (\log n)^{(k_1+k_2-1+o(1))n}.
\end{equation}

{\bf (ii)\/} Lower bound. Let $k_1\ge k_2$. We start with
\[
I_n\ge \Bbb E\Biggl[\Bbb I\bigl(\max_iX_i\le e^{-1}),\,\max_j Y_j\le\be\bigr)\prod_{1\le i\neq j\le n}\Bigl(1-X_i\, Y_j\Bigr)\Biggr],
\]
the constraint on $\max_i X_i$ being dictated by log-convexity of $\phi_{k_1}(x)$ for $x\in (0,e^{-1}]$. As for the constraint on $\max_jY_j$, we need it to have the bound
\[
1-Y_j X_i\ge e^{-\a X_iY_j},\, \qquad \a>1,\, \,\be:= 1-\hat\a/\a.
\]
So, with $s:=\sum_{i\in [n]} x_i$,
\[
I_n\ge\idotsint\limits_{\boldsymbol x\in [0,e^{-1}]^n}\Biggl(\int_0^{\be} e^{-\a sy} \phi_{k_2}(y)\,dy\Biggr)^n\bigl(\phi_{k_1}(s/n)\bigr)^n\,\,d\boldsymbol x.
\]
Since the density of the sum of $n$ independent Uniforms is $s^{n-1}/(n-1)!$ if $s\le 1$, the last bound yields
\begin{align*}
I_n&\ge \Biggl(\int_0^{\be} e^{-\a y} \phi_{k_2}(y)\,dy\Biggr)^n \left(\frac{\log^{k_1-1}n}{(k_1-1)!}\right)^n\idotsint\limits_{\boldsymbol x\in [0,e^{-1}]^n\atop s\le 1} d\bold x\\
&\ge\ga^n (\log n)^{(k_1-1)n}\int_0^{e^{-1}}\!\!\!\frac{s^{n-1}}{(n-1)!}\,ds=\ga_1^n\, \frac{(\log n)^{(k_1-1)n}}{n!}.
\end{align*}
Therefore
\begin{equation}\label{wlower}
\Bbb E[S_{n,w}] =n!\,I_n\ge \ga_2^n\, (\log n)^{(k_1-1)n}.
\end{equation} 
Combining \eqref{wupper} and \eqref{wlower} we have proved
\begin{theorem}\label{weak} If $\max(k_1,k_2)>1$ then
\[
(\log n)^{n\bigl(\max(k_2,k_1)-1-o(1)\bigr)}\le\Bbb E[S_{n,w}]\le (\log n)^{(k_1+k_2-1+o(1))n},
\]
i.e. $\Bbb E[S_{n,w}]$ grows super-exponentially fast.
\end{theorem}

\end{document}